\documentclass{amsart}
\usepackage{amsmath,amssymb,latexsym}
\usepackage[margin=1in, marginparwidth=0.8in]{geometry}
\usepackage{tikz}

\newtheorem{theorem}{Theorem}
\newtheorem{definition}[theorem]{Definition}
\newtheorem{corollary}[theorem]{Corollary}
\newtheorem{example}[theorem]{Example}
\newtheorem{lemma}[theorem]{Lemma}
\newtheorem{proposition}[theorem]{Proposition}
\newtheorem{remark}[theorem]{Remark}

\newcommand{\bepsilon}{\boldsymbol{\epsilon}}
\newcommand{\be}{\boldsymbol{e}}
\newcommand{\bi}{\boldsymbol{i}}
\newcommand{\SL}{\operatorname{SL}}
\newcommand{\ZZ}{\mathbb{Z}}

\author[Demonet]{Laurent Demonet}
\address[Laurent Demonet]{Graduate School of Mathematics, Nagoya University, Furocho, Chikusaku, 464-8602 Nagoya, Japan}
\email{Laurent.Demonet@normalesup.org}
\thanks{The first author was partially supported by JSPS Grant-in-Aid for Young Scientist (B) 26800008.}

\author[Plamondon]{Pierre-Guy Plamondon}
\address[Pierre-Guy Plamondon]{Laboratoire de Math\'ematiques d'Orsay, Univ. Paris-Sud, CNRS, Univ.
Paris-Saclay, 91405 Orsay, France.}
\email{pierre-guy.plamondon@math.u-psud.fr}
\thanks{The second author was partially supported by the French ANR grant SC3A (ANR-15-CE40-0004-01).}

\author[Rupel]{Dylan Rupel}
\address[Dylan Rupel]{Department of Mathematics, University of Notre Dame, Notre Dame, Indiana 46556, USA.}
\email{drupel@nd.edu}

\author[Stella]{Salvatore Stella}
\address[Salvatore Stella]{IN$d$AM - Marie Curie Actions fellow, Universit\`a ``La Sapienza'', Roma, Italy.}
\email{stella@mat.uniroma1.it}

\author[Tumarkin]{Pavel Tumarkin}
\address[Pavel Tumarkin]{Department of Mathematical Sciences, Durham University, South Road, Durham, DH1 3LE, UK.}
\email{pavel.tumarkin@durham.ac.uk}

\title{$\SL_2$-Tilings Do Not Exist in Higher Dimensions (mostly)}
\subjclass[2010]{05E15, 13F60}

\begin{document}
\begin{abstract}
  We define a family of generalizations of $\SL_2$-tilings to higher dimensions called $\bepsilon$-$\SL_2$-tilings.  
  We show that, in each dimension 3 or greater, $\bepsilon$-$\SL_2$-tilings exist only for certain choices of $\bepsilon$.  
  In the case that they exist, we show that they are essentially unique and have a concrete description in terms of odd Fibonacci numbers.
\end{abstract}
\maketitle

\section{$\SL_2$-Tilings of the Plane}
  The aim of this note is to study higher-dimensional analogues of the following object.
  \begin{definition}[\cite{AssemReutenauerSmith}]
    A bi-infinite array $(a_{ij})_{i,j\in\ZZ}$ with $a_{ij}\in\ZZ_{>0}$ is called an \emph{$\SL_2$-tiling of $\ZZ^2$} if the entries satisfy the relation
    \begin{equation}\label{eq:sl2 recursion}
      a_{i,j+1}a_{i+1,j}-a_{ij}a_{i+1,j+1}=1.
    \end{equation}
    A bi-infinite array $(b_{ij})_{i,j\in\ZZ}$ with $b_{ij}\in\ZZ_{>0}$  is called an \emph{anti-$\SL_2$-tiling of $\ZZ^2$} if the entries satisfy the relation
    \begin{equation}\label{eq:anti-sl2 recursion}
      b_{i,j+1}b_{i+1,j}-b_{ij}b_{i+1,j+1}=-1.
    \end{equation}
  \end{definition}
  The notion of an anti-$\SL_2$-tiling is not actually giving anything new as shown by the following lemma, however this notion will be useful for our considerations in higher dimensions.
  \begin{lemma}
    If $(a_{ij})_{i,j\in\ZZ}$ is an $\SL_2$-tiling, then taking $b_{ij}=a_{i,-j}$ gives an anti-$\SL_2$-tiling.
  \end{lemma}
  One should think of the difference between $\SL_2$-tilings and anti-$\SL_2$-tilings as viewing the lattice $\ZZ^2$ ``from above'' or ``from below.''  
  The following result from \cite{AssemReutenauerSmith} was our starting point.
  \begin{theorem}[\cite{AssemReutenauerSmith}]
    There exist infinitely many $\SL_2$-tilings of $\ZZ^2$.
  \end{theorem}

  In fact, it is shown in \cite{AssemReutenauerSmith} that any admissible frontier of $1$'s in the lattice, can be completed into a unique $\SL_2$-tiling.  
  An interpretation of all possible $\SL_2$-tilings was later given in \cite{BessenrodtHolmJorgensen} in terms of triangulations of a polygon with infinitely many vertices.

  The following anti-$\SL_2$-tiling will be relevant in our higher dimensional analysis.
  We will call it the \emph{staircase} anti-$\SL_2$-tiling of $\ZZ^2$.
  \begin{example}\label{ex:Fibonacci}
    Consider the anti-$\SL_2$-tiling $(a_{ij})_{i,j\in\ZZ}$ of $\ZZ^2$ with $a_{ij}=1$ if $i+j\in\{0,1\}$.  
    Using \eqref{eq:anti-sl2 recursion} and the well-known recursion $F_{2r-1}F_{2r+3}=F_{2r+1}^2+1$ $(r\ge1)$ for the odd Fibonacci numbers, it is easy to see that
      \[
        a_{ij}
        =
        \begin{cases}
          F_{2r-1} & \text{if $i+j=r\ge1$;}\\
          F_{-2r+1} & \text{if $i+j=r\le0$;}
        \end{cases}
      \]
      where we number the Fibonacci numbers as:
      \[
        \begin{tabular}{|c|c|c|c|c|c|c|c} 
          $F_1$ & $F_2$ & $F_3$ & $F_4$ & $F_5$ & $F_6$ & $F_7$ & $\cdots$\\
          \hline 1 & 1 & 2 & 3 & 5 & 8 & 13 & $\cdots$
        \end{tabular}
      \]
      The following figure is a portion of this tiling. 
      Note the bolded frontier of $1$'s; it is an ``infinite staircase''.
      \begin{center}
        \begin{tikzpicture}
          \draw (0,1) node{${\bf 1}$} (.75,1) node{${\bf 1}$} (1.5,1) node{$2$} (2.25,1) node{$5$} (3,1) node{$13$} (3.75,1) node{$34$} (4.5,1) node{$89$} (5.25,1) node{$233$};
          \draw (0,.5) node{$2$} (.75,.5) node{${\bf 1}$} (1.5,0.5) node{${\bf 1}$} (2.25,0.5) node{$2$} (3,0.5) node{$5$} (3.75,0.5) node{$13$} (4.5,0.5) node{$34$} (5.25,0.5) node{$89$};
          \draw (0,0) node{$5$} (.75,0) node{$2$} (1.5,0) node{${\bf 1}$} (2.25,0) node{${\bf 1}$} (3,0) node{$2$} (3.75,0) node{$5$} (4.5,0) node{$13$} (5.25,0) node{$34$};
          \draw (0,-.5) node{$13$} (.75,-.5) node{$5$} (1.5,-.5) node{$2$} (2.25,-.5) node{${\bf 1}$} (3,-.5) node{${\bf 1}$} (3.75,-.5) node{$2$} (4.5,-.5) node{$5$} (5.25,-.5) node{$13$};
          \draw (0,-1) node{$34$} (.75,-1) node{$13$} (1.5,-1) node{$5$} (2.25,-1) node{$2$} (3,-1) node{${\bf 1}$} (3.75,-1) node{${\bf 1}$} (4.5,-1) node{$2$} (5.25,-1) node{$5$};
          \draw (0,-1.5) node{$89$} (.75,-1.5) node{$34$} (1.5,-1.5) node{$13$} (2.25,-1.5) node{$5$} (3,-1.5) node{$2$} (3.75,-1.5) node{${\bf 1}$} (4.5,-1.5) node{${\bf 1}$} (5.25,-1.5) node{$2$};
          \draw (0,-2) node{$233$} (.75,-2) node{$89$} (1.5,-2) node{$34$} (2.25,-2) node{$13$} (3,-2) node{$5$} (3.75,-2) node{$2$} (4.5,-2) node{${\bf 1}$} (5.25,-2) node{${\bf 1}$};
          \draw (0,-2.5) node{$610$} (.75,-2.5) node{$233$} (1.5,-2.5) node{$89$} (2.25,-2.5) node{$34$} (3,-2.5) node{$13$} (3.75,-2.5) node{$5$} (4.5,-2.5) node{$2$} (5.25,-2.5) node{${\bf 1}$};
        \end{tikzpicture}
      \end{center}
    \end{example}

\section{$\SL_2$-Tilings in Higher Dimensions}
  Denote integer vectors by $\bi=(i_1\dots,i_n)$ and by $\be_k$ the $k$-th unit vector.
  A \emph{signature matrix} is a symmetric $n\times n$ matrix $\bepsilon=(\epsilon_{k\ell})$ with $\epsilon_{k\ell}=\pm 1$ whenever $k\neq \ell$ and $\epsilon_{kk}=-1$. 
  \begin{definition}
    Fix a signature matrix $\bepsilon$.  
    An array $(a_{\bi})_{\bi\in\ZZ^n}$ with $a_{\bi}\in\ZZ_{>0}$ is called an \emph{$\bepsilon-\SL_2$-tiling of $\ZZ^n$} if for each $k \neq\ell$ we have
    \begin{equation}\label{eq:higher sl2 recursion}
      a_{\bi+\be_\ell}a_{\bi+\be_k}-a_{\bi}a_{\bi+\be_k+\be_\ell}=\epsilon_{k\ell}.
    \end{equation}
  \end{definition}
  The requirement on the diagonal entries of signature matrices might seem arbitrary right now because they do not play any role in the above definition; we will see later on that it is indeed a consistent choice.
  
  The situation is now different than the $n=2$ case, all the $\bepsilon$-$\SL_2$-tilings are not necessarily equivalent, however there do remain relations among them.
  \begin{lemma}\label{le:flip}
    Let $\bepsilon=(\epsilon_{k\ell})$ be any signature matrix and write $\bepsilon^{(r)}$ for the matrix obtained from $\bepsilon$ by changing the sign of all the entries in row $r$ and column $r$, leaving the diagonal entries fixed. 
    That is, $\bepsilon^{(r)} =(\epsilon'_{k\ell})$ where $\epsilon'_{k\ell}=-\epsilon_{k\ell}$ if exactly one of $k$ and $\ell$ equals $r$ and $\epsilon'_{k\ell}=\epsilon_{k\ell}$ otherwise. 
    If $(a_{\bi})_{\bi\in\ZZ^n}$ is an $\bepsilon-\SL_2$-tiling, then taking $b_{\bi}=a_{\bi-2i_r\be_r}$ gives an $\bepsilon^{(r)}-\SL_2$-tiling.
  \end{lemma}

  \begin{definition}
    If $\bepsilon$ is a signature matrix such that $\epsilon_{k\ell}=1$ (resp. $\epsilon_{k\ell}=-1$) whenever $k\neq\ell$,
    we refer to an $\bepsilon$-$\SL_2$-tiling as an \emph{$\SL_2$-tiling} (resp. \emph{anti-$\SL_2$-tiling}) of $\ZZ^n$.
  \end{definition}

  \begin{lemma}\label{le:constant slices}
    Let $n\ge 3$ and assume $(a_{\bi})_{\bi\in\ZZ^n}$ is either an $\SL_2$-tiling or an anti-$\SL_2$-tiling of $\ZZ^n$.
    Then for any $r\in\ZZ$ the set $\{a_{\bi}: \sum_{j=1}^n i_j=r\}$ consists of a single element.
  \end{lemma}
  \begin{proof}
    Pick any three distinct indices $j,k,\ell\in[1,n]$.
    To prove our claim we compute $a_{\bi+\be_j+\be_k+\be_\ell}$ in terms of $a_{\bi}, a_{\bi+\be_j}, a_{\bi+\be_k}, a_{\bi+\be_\ell}$ in three different ways.
    For simplicity of notation we set:
    \[
      \epsilon_{jk}=\epsilon_{j\ell}=\epsilon_{k\ell}=\epsilon,
      \quad\quad 
      a_{\bi}=a,
      \quad\quad 
      a_{\bi+\be_j}=x,
      \quad\quad 
      a_{\bi+\be_k}=y,
      \quad\quad 
      a_{\bi+\be_\ell}=z.
    \]
    The following picture will be useful.
    
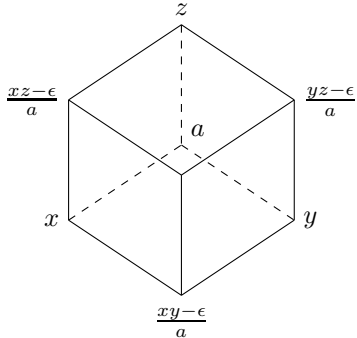
\begin{center}
      \begin{tikzpicture}
        \draw (0,0) -- ++(1.5,1)  node[above] {$z$} ;
        \draw (0,0) node[left] {$\frac{xz-\epsilon}{a}$} ;
        \draw (0,0) -- ++(0,-1.6) node[left] {$x$} ;
        \draw [dashed] (1.5,1) -- ++(0,-1.6) node[above right] {$a$};
        \draw [dashed] (0,-1.6) -- ++(1.5,1) ;
        \draw (0,0) -- ++(1.5,-1); 
        \draw (1.5,1) -- ++(1.5,-1) node[right] {$\frac{yz-\epsilon}{a}$};
        \draw (0,-1.6) -- ++(1.5,-1) node[below]{$\frac{xy-\epsilon}{a}$};
        \draw [dashed] (1.5,-.6) -- ++(1.5,-1);
        \draw (1.5,-1) -- ++(0,-1.6) ;
        \draw (1.5,-1) -- ++(1.5,1) ;
        \draw (1.5,-2.6) -- ++(1.5,1) ;
        \draw (3,0) -- ++(0,-1.6) node[right]{$y$};
      \end{tikzpicture}
    \end{center}
    Using \eqref{eq:higher sl2 recursion} three times we get
    \[
      a_{\bi+\be_j+\be_k}=\frac{xy-\epsilon}{a},
      \quad\quad 
      a_{\bi+\be_k+\be_\ell}=\frac{yz-\epsilon}{a},
      \quad\quad 
      a_{\bi+\be_j+\be_\ell}=\frac{xz-\epsilon}{a}.
    \]
    Then applying \eqref{eq:higher sl2 recursion} three more times gives 
    \begin{equation*}
      a_{\bi+\be_j+\be_k+\be_\ell}=\left\{\begin{array}{l}
        \frac{a_{\bi+\be_j+\be_k}a_{\bi+\be_j+\be_\ell}-\epsilon}{a_{\bi+\be_j}}=\frac{xyz}{a^2}-\epsilon\frac{y+z}{a^2}-\epsilon\frac{a^2-\epsilon}{a^2x}\\
        \frac{a_{\bi+\be_j+\be_k}a_{\bi+\be_k+\be_\ell}-\epsilon}{a_{\bi+\be_k}}=\frac{xyz}{a^2}-\epsilon\frac{x+z}{a^2}-\epsilon\frac{a^2-\epsilon}{a^2y}\\
        \frac{a_{\bi+\be_j+\be_\ell}a_{\bi+\be_k+\be_\ell}-\epsilon}{a_{\bi+\be_\ell}}=\frac{xyz}{a^2}-\epsilon\frac{x+y}{a^2}-\epsilon\frac{a^2-\epsilon}{a^2z}\end{array}\right.
    \end{equation*}
    It follows that $\frac{x-y}{a^2}=\frac{a^2-\epsilon}{a^2x}-\frac{a^2-\epsilon}{a^2y}$ or $(xy+a^2-\epsilon)(x-y)=0$.  
    But $xy+a^2-\epsilon\ge1$ since $a,x,y\ge1$, hence $x=y$.  
    Similarly $y=z$.
    The result then follows by iterating on all possible triples of distinct indices.
  \end{proof}

  We now come to our first main result: in dimension $n$, an ``infinite staircase'' of $1$'s yields the only possible anti-$\SL_2$-tiling.

  \begin{theorem} 
    \label{thm:antitiling}
    For $n\ge3$, there exists a unique (up to translation) anti-$\SL_2$-tiling of $\ZZ^n$. 
    Any of its ``two dimensional slices'' obtained by fixing all but two of the coordinates of $\bi$ is a translation of the staircase anti-$\SL_2$-tiling of $\ZZ^2$ from Example \ref{ex:Fibonacci}.
    In particular, all the integers appearing are odd Fibonacci numbers.
  \end{theorem}  
  \begin{proof}
    Assume $(a_{\bi})_{\bi\in\ZZ^n}$ is a anti-$\SL_2$-tiling of $\ZZ^n$.  
    Pick $\bi$ with $a_{\bi}$ minimal.  
    Applying \eqref{eq:higher sl2 recursion} gives
    \[
      a_{\bi+\be_1}a_{\bi-\be_2}=a_{\bi}a_{\bi+\be_1-\be_2}+1=a_{\bi}^2+1,
    \]
    where we applied Lemma~\ref{le:constant slices} in the last equality.
    If $a_{\bi}>1$, this implies $a_{\bi+\be_1}<a_{\bi}$ or $a_{\bi-\be_2}<a_{\bi}$, contradicting minimality, so we must have $a_{\bi}=1$.
    In turn, again leveraging Lemma~\ref{le:constant slices}, this implies $\{a_{\bi+\be_k},a_{\bi-\be_k}\}=\{1,2\}$.
    Without loss of generality we will assume $a_{\bi+\be_k}=2$ and $\sum_{j=1}^n i_j=1$.
    Then applying \eqref{eq:higher sl2 recursion} repeatedly shows that $a_{\bi'}$ with $\sum_{j=1}^n i'_j=r\ge1$ is exactly the $r^{th}$ odd Fibonacci number $F_{2r-1}$ (see Example~\ref{ex:Fibonacci}).
    Similarly one sees that $a_{\bi'}$ with $\sum_{j=1}^n i'_j =r\le0$ is the odd Fibonacci number $F_{-2r+1}$.
  \end{proof}

  \begin{proposition}\label{pr:nonexistence}
    There does not exist any $\SL_2$-tiling of $\ZZ^n$ for $n\geq3$.
  \end{proposition}
  \begin{proof}
    It suffices to show that there is no $\SL_2$-tiling of $\ZZ^3$.
    Assume $(a_{\bi})_{\bi\in\ZZ^3}$ is an $\SL_2$-tiling of $\ZZ^3$.
    Pick $\bi$ with $a_{\bi}$ minimal.
    Applying \eqref{eq:higher sl2 recursion} gives
    \[
      a_{\bi+\be_1}a_{\bi-\be_2}=a_{\bi}a_{\bi+\be_1-\be_2}-1=a_{\bi}^2-1,
    \]
    where we applied Lemma~\ref{le:constant slices} in the last equality.
    But this implies $a_{\bi+\be_1}<a_{\bi}$ or $a_{\bi-\be_2}<a_{\bi}$, contradicting minimality.
  \end{proof}

  \begin{corollary}\label{co:n=3}
    For $n=3$, there are precisely $4$ signature matrices $\bepsilon$ for which there exists an $\bepsilon$-$\SL_2$-tiling. 
    For such $\bepsilon$, this $\bepsilon$-$\SL_2$-tiling is unique (up to translation).
    More precisely, an $\bepsilon$-$\SL_2$-tiling exists if and only if $\epsilon_{12}\epsilon_{13}\epsilon_{23}=-1$.
  \end{corollary}
  \begin{proof}
    The claim follows immediately from the observation that any signature matrix for $n=3$ is either one of the two satisfying $\epsilon_{12}=\epsilon_{13}=\epsilon_{23}$ or is obtained from one of these with a single application of Lemma~\ref{le:flip}.
  \end{proof}

  We are finally ready to classify all $\bepsilon$-$\SL_2$-tilings for any $n\geq3$.
  \begin{theorem}
    \label{thm:classification}
    For $n\geq3$, there are precisely $2^{n-1}$ signature matrices $\bepsilon$ for which there exists an $\bepsilon$-$\SL_2$-tiling of $\ZZ^n$.
    They are precisely the signature matrices obtainable from the anti-$\SL_2$-signature matrix by repeated application of Lemma~\ref{le:flip}.
    Whenever an $\bepsilon$-$\SL_2$-tiling exists, it is unique up to translation.
  \end{theorem}
  \begin{proof}
    Let $(a_{\bi})_{\bi\in\ZZ^n}$ be an $\bepsilon$-$\SL_2$-tiling of $\ZZ^n$. 
    Fixing all but any three distinct entries of $\bi$ gives a tiling of $\ZZ^3$.
    Therefore, it follows from Corollary \ref{co:n=3} that we have an inclusion $E \subset E'$, where $E$ is the set of $n\times n$ signature matrices $\bepsilon$ which admit an $\bepsilon$-$\SL_2$-tiling, and $E'$ is the set of $n\times n$ signature matrices $\bepsilon$ satisfying $\epsilon_{jk} \epsilon_{k\ell} \epsilon_{j\ell} = -1$ for any triple of distinct indices $j,k,\ell$.

    Any row (or equivalently any column) of a matrix $\bepsilon$ in $E'$ determines uniquely all the remaining entries of $\bepsilon$, therefore $E'$ is in bijection with $\{\pm 1\}^{n-1}$ and $\#E'= 2^{n-1}$.

    Using Lemma \ref{le:flip}, there is an action of $(\ZZ/2\ZZ)^{n-1}$ on $E$ given by $\bepsilon \mapsto \bepsilon^{(r)}$ for $1 \leq r \leq n-1$. 
    This action is free; indeed the only element of $(\ZZ/2\ZZ)^{n-1}$ leaving invariant the last column of any given matrix of $E$ is the identity. 
    Thanks to Theorem \ref{thm:antitiling}, $E$ is not empty and so we compute $\#E \geq 2^{n-1} = \# E' \geq \#E$ and deduce that $E = E'$.

    The uniqueness claim also follows immediately from Corollary \ref{co:n=3} by fixing all but any three distinct entries of $\bi$.

  \end{proof}

  Note that the claim of Theorem \ref{thm:classification} could be rephrased by saying that, up to fixing the origin and choosing the orientation of each of the coordinate axes, there is a unique tiling of $\ZZ^n$ for $n\ge3$.
  
  \begin{remark}
    It is now clear why we choose the diagonal entries of $\bepsilon$ to be equal to $-1$: any $\bepsilon$-$\SL_2$-tiling consists of odd Fibonacci numbers and \eqref{eq:higher sl2 recursion} is satisfied also for $k=\ell$.
  \end{remark} 

\section*{Acknowledgements}
  These results were obtained while we were taking part in the Conference on Cluster Algebras and Representation Theory at the KIAS in Seoul, South Korea. 
  We would like to thank the organizers of that meeting for the stimulating environment they provided.
  
  We are especially grateful to Peter J\o rgensen for his very inspiring talk at the same conference: the results presented here are the content of the exciting discussion he sparked.

\end{document}